\documentclass[11pt,a4paper]{article}

\usepackage{amsmath,amsfonts,amssymb,amsthm,latexsym}
\usepackage{graphics,epsfig,color,enumerate}
\usepackage{setspace,mathtools,enumitem,dsfont,verbatim}
\usepackage{caption}
\usepackage{subcaption}
\usepackage{todonotes}
\usepackage[english]{babel}
\makeatletter
\def\@captype{figure}
\makeatother

\newtheorem{theorem}{Theorem}[section]
\newtheorem{lemma}[theorem]{Lemma}
\newtheorem{proposition}[theorem]{Proposition}

\newtheorem{remark}[theorem]{Remark}

\newtheorem{definition}[theorem]{Definition}
\newtheorem{cond}[theorem]{Condition}

\numberwithin{equation}{section}

\newcommand{\N}{\mathbb{N}}
\newcommand{\R}{\mathbb{R}}
\newcommand{\Conf}{\emph{\text{Conf}}}
\newcommand{\sss}{\scriptscriptstyle}
\newcommand{\TV}{\sss \mathrm{TV}}

\newcommand{\tmix}{t_{\mathrm{mix}}}

\newcommand{\eqan}[1]{\begin{align} #1 \end{align}}
\newcommand{\nn}{\nonumber}
\newcommand{\expec}{\mathbb{E}}
\newcommand{\prob}{\mathbb{P}}

\newcommand{\SP}{\mathsf{SP}}

\newcommand{\G}{\mathsf{G}}

\newcommand{\dist}{\mathrm{dist}}

\title{Mixing times of random walks \\ on dynamic configuration models}

\author{
Luca Avena
\footnotemark[1]
\\

Hakan G\"{u}lda\c{s}
\footnotemark[1]
\\

Remco van der Hofstad
\footnotemark[2]
\\

Frank den Hollander
\footnotemark[1]
}

\footnotetext[1]{
Mathematical Institute, Leiden University, P.O.\ Box 9512,
2300 RA Leiden, The Netherlands
}
\footnotetext[2]{
Department of Mathematics and Computer Science, Eindhoven University of Technology, 
P.O.\ Box 513, 5600 MB Eindhoven, The Netherlands
}

\date{\today}

\begin{document}

\maketitle 

\begin{abstract}
The mixing time of a random walk, with or without backtracking, on a random graph generated 
according to the configuration model on $n$ vertices, is known to be of order $\log n$. In this 
paper we investigate what happens when the random graph becomes {\em dynamic}, namely, 
at each unit of time a fraction $\alpha_n$ of the edges is randomly rewired. Under mild conditions 
on the degree sequence, guaranteeing that the graph is locally tree-like, we show that for every 
$\varepsilon\in(0,1)$ the $\varepsilon$-mixing time of random walk without backtracking grows 
like $\sqrt{2\log(1/\varepsilon)/\log(1/(1-\alpha_n))}$ as $n \to \infty$, provided that $\lim_{n\to\infty} 
\alpha_n(\log n)^2=\infty$. The latter condition corresponds to a regime of fast enough graph 
dynamics. Our proof is based on a randomised stopping time argument, in combination with 
coupling techniques and combinatorial estimates. The stopping time of interest is the first time 
that the walk moves along an edge that was rewired before, which turns out to be close to a 
strong stationary time.
\end{abstract}



\medskip\noindent 
\emph{Acknowledgements:} 
The research in this paper is supported by the Netherlands Organisation for Scientific 
Research (NWO Gravitation Grant NETWORKS-024.002.003). The work of RvdH is
further supported by NWO through VICI grant 639.033.806. The authors are grateful to 
Perla Sousi and Sam Thomas for pointing out an error in an earlier draft of the paper.

\newpage


\section{Introduction and main result}


\subsection{Motivation and background}

The \emph{mixing time} of a Markov chain is the time it needs to approach its stationary 
distribution. For random walks on \emph{finite graphs}, the characterisation of the mixing 
time has been the subject of intensive study. One of the main motivations is the fact that 
the mixing time gives information about the geometry of the graph (see the books by 
Aldous and Fill~\cite{AF} and by Levin, Peres and Wilmer~\cite{LPW} for an overview 
and for applications).  Typically, the random walk is assumed to be `simple', meaning 
that steps are along edges and are drawn uniformly at random from a set of allowed 
edges, e.g.\ with or without backtracking.

In the last decade, much attention has been devoted to the analysis of mixing times for 
random walks on \emph{finite random graphs}. Random graphs are used as models for 
real-world networks. Three main models have been in the focus of attention: (1) the 
Erd\H{o}s-R\'enyi random graph (Benjamini, Kozma and Wormald~\cite{BKW}, Ding, 
Lubetzky and Peres~\cite{DLP}, Fountoulakis and Reed~\cite{FR}, Nachmias and 
Peres~\cite{NP}); (2) the configuration model (Ben-Hamou and Salez~\cite{B-HS},
Berestycki, Lubetzky, Peres and Sly~\cite{BLPS}, Bordenave, Caputo and Salez~\cite{BCS}, 
Lubetzky and Sly~\cite{LS}); (3) percolation clusters (Benjamini and Mossel~\cite{BM}).

Many real-world networks are dynamic in nature. It is therefore natural to study random 
walks on \emph{dynamic finite random graphs}. This line of research was initiated recently
by Peres, Stauffer and Steif~\cite{PSS} and by Peres, Sousi and Steif~\cite{PSS2}, who 
characterised the mixing time of a simple random walk on a dynamical percolation cluster 
on a $d$-dimensional discrete torus, in various regimes. The goal of the present paper is 
to study the mixing time of a random walk \emph{without backtracking} on a dynamic version 
of the configuration model.

The \emph{static} configuration model is a random graph with a prescribed degree sequence 
(possibly random). It is popular because of its mathematical tractability and its flexibility in 
modeling real-world networks (see van der Hofstad~\cite[Chapter 7]{RvdH1} for an overview). 
For random walk on the static configuration model, with or without backtracking, the asymptotics 
of the associated mixing time, and related properties such as the presence of the so-called 
cutoff phenomenon, were derived recently by Berestycki, Lubetzky, Peres and Sly~\cite{BLPS}, 
and by Ben-Hamou and Salez~\cite{B-HS}. In particular, under mild assumptions on the 
degree sequence, guaranteeing that the graph is an \emph{expander} with high probability, 
the mixing time was shown to be of order $\log n$, with $n$ the number of vertices. 

In the present paper we consider a \emph{discrete-time dynamic} version of the configuration 
model, where at each unit of time a fraction $\alpha_n$ of the edges is sampled and rewired 
uniformly at random. [A different dynamic version of the configuration model was 
considered in the context of \emph{graph sampling}. See Greenhill~\cite{G} and references 
therein.] Our dynamics \emph{preserves the degrees} of the vertices. Consequently, when 
considering a random walk on this dynamic configuration model, its \emph{stationary distribution 
remains constant over time} and the analysis of its mixing time is a well-posed question. It 
is natural to expect that, due to the graph dynamics, the random walk \emph{mixes 
faster} than the $\log n$ order known for the static model. In our main theorem we will 
make this precise under mild assumptions on the prescribed degree sequence stated in 
Condition~\ref{cond-degree-reg} and Remark~\ref{conditions} below. By requiring that 
$\lim_{n\to\infty} \alpha_n(\log n)^2=\infty$, which corresponds to a regime of fast enough 
graph dynamics, we find in Theorem~\ref{thm:mainthm} below that for every $\varepsilon
\in(0,1)$ the $\varepsilon$-mixing time for random walk \emph{without backtracking} grows like 
$\sqrt{2\log(1/\varepsilon)/\log(1/(1-\alpha_n))}$ as $n\to\infty$, with high probability in the 
sense of Definition~\ref{def:whp} below. Note that this mixing time is $o(\log n)$, so that 
the dynamics indeed speeds up the mixing. 


\subsection{Model}
We start by defining the model and setting up the notation. The set of vertices is denoted by 
$V$ and the degree of a vertex $v\in V$ by $d(v)$. Each vertex $v\in V$ is thought of as being 
incident to $d(v)$ \emph{half-edges} (see Fig.~\ref{fig:halfedge}). We write $H$ for the set of 
half-edges, and assume that each half-edge is associated to a vertex via incidence. We denote 
by $v(x)\in V$ the vertex to which $x\in H$ is incident and by $H(v) \coloneqq \{x\in H\colon\,v(x)=v \}
\subset H$ the set of half-edges incident to $v\in V$. If $x,y\in H(v)$ with $x\neq y$, then we 
write $x\sim y$ and say that $x$ and $y$ are siblings of each other. The degree of a half-edge 
$x\in H$ is defined as
\begin{equation}
\label{degdef}
\deg(x) \coloneqq d(v(x))-1.
\end{equation}
We consider graphs on $n$ vertices, i.e., $|V| = n$, with $m$ edges, so that $|H|=\sum_{v\in V} 
\deg(v) = 2m \eqqcolon \ell$.

\begin{figure}[htbp]
\centering
\vspace{0.5cm}
\includegraphics[width=0.25\textwidth]{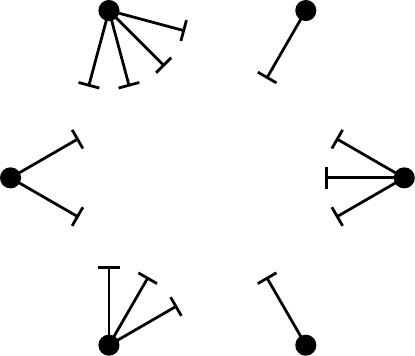}
\vspace{0.5cm}
\caption{\small Vertices with half-edges.}
\label{fig:halfedge}
\end{figure}

The \emph{edges} of the graph will be given by a \emph{configuration} that is a \emph{pairing of 
half-edges}. We denote by $\eta(x)$ the half-edge paired to $x\in H$ in the configuration $\eta$.
A configuration $\eta$ will be viewed as a bijection of $H$ without fixed points and with the property 
that $\eta(\eta(x)) = x$ for all $x\in H$ (also called an involution). With a slight abuse of notation, we 
will use the same symbol $\eta$ to denote the set of pairs of half-edges in $\eta$, so $\{x,y\}\in\eta$ 
means that $\eta(x) = y$ and $\eta(y) = x$. Each pair of half-edges in $\eta$ will also be called an 
edge. The set of all configurations on $H$ will be denoted by $\Conf_H$.

We note that each configuration gives rise to a graph that may contain self-loops (edges having 
the same vertex on both ends) or multiple edges (between the same pair of vertices). On the 
other hand, a graph can be obtained via several distinct configurations.

We will consider asymptotic statements in the sense of $|V| = n \to \infty$. Thus, quantities like
$V, H, d, \deg$ and $\ell$ all depend on $n$. In order to lighten the notation, we often suppress 
$n$ from the notation.


\subsubsection{Configuration model}
\label{SCM}

We recall the definition of the configuration model, phrased in our notation. Inspired by Bender 
and Canfield~\cite{BenCan78}, the configuration model was introduced by Bollob\'as~\cite{Boll80b} 
to study the number of regular graphs of a given size (see also Bollob\'as~\cite{Boll01}). Molloy and 
Reed~\cite{MolRee95},~\cite{MolRee98} introduced the configuration model with general prescribed 
degrees.

The configuration model on $V$ with degree sequence $(d(v))_{v\in V}$ is the uniform 
distribution on $\Conf_H$. We sometimes write $d_n=(d(v))_{v\in V}$ when we wish to 
stress the $n$-dependence of the degree sequence. Identify $H$ with the set 
$$
[1,\ell] \coloneqq \{1,\dots,\ell\}.
$$ 
A sample $\eta$ from the configuration model can be generated by the following 
\emph{sampling algorithm}:

\begin{itemize}
\item[1.] 
Initialize $U = H, \eta = \varnothing$, where $U$ denotes the set of unpaired half-edges.
\item[2.] Pick a half-edge, say $x$, uniformly at random from $U\setminus\{\min U\}$. 
\label{item:cmgen_two}
\item[3.]
Update $\eta \to \eta \cup \{\{x,\min U\}\}$ and $U \to U\setminus\{x,\min U\}$.
\item[4.]
If $U\neq\varnothing$, then continue from step 2. Else return $\eta$.
\end{itemize}
The resulting configuration $\eta$ gives rise to a graph on $V$ with degree sequence 
$(d(v))_{v\in V}$.

\begin{remark}
\label{multigraph}
{\rm Note that in the above algorithm two half-edges that belong to the same vertex can be 
paired, which creates a self-loop, or two half-edges that belong to vertices that already 
have an edge between them can be paired, which creates multiple edges. However, if 
the degrees are not too large (as in Condition~\ref{cond-degree-reg} below), then as 
$n\to\infty$ the number of self-loops and the number of multiple edges converge to two 
independent Poisson random variables (see Janson~\cite{Jans06b},~\cite{Jans13a}, 
Angel, van der Hofstad and Holmgren~\cite{AngHofHol16}). Consequently, convergence 
in probability for the configuration model implies convergence in probability for the configuration 
model conditioned on being simple.}
\end{remark}

Let $U_n$ be uniformly distributed on $[1,n]$. Then
\begin{equation}
\label{degree}
D_n = d(U_n)
\end{equation} 
is the degree of a random vertex on the graph of size $n$.  Write $\mathbb{P}_n$ to 
denote the law of $D_n$. Throughout the sequel, we impose the following mild regularity 
conditions on the degree sequence:
\begin{cond}
\label{cond-degree-reg} {\bf (Regularity of degrees)}
\begin{itemize}
\item[{\rm (R1)}]
Let $\ell = |H|$. Then $\ell$ is even and of order $n$, i.e., $\ell=\Theta(n)$ as $n\to\infty$.
\item [{\rm (R2)}]
Let 
\begin{equation}
\nu_n \coloneqq \frac{\sum_{z\in H}\deg(z)}{\ell}=\frac{\sum_{v\in V} d(v)[d(v)-1]}{\sum_{v\in V} d(v)}
= \frac{\mathbb{E}_n(D_n(D_n-1))}{\mathbb{E}_n(D_n)}
\end{equation} 
denote the expected degree of a uniformly chosen half-edge. Then $\limsup_{n\to\infty} \nu_n < \infty$.
\item [{\rm (R3)}] 
$\mathbb{P}_n(D_n \geq 2) = 1$ for all $n\in\mathbb{N}$.
\end{itemize}
\end{cond}

\begin{remark}
\label{conditions}
{\rm Conditions (R1) and (R2) are minimal requirements to guarantee that the graph is locally tree-like
(in the sense of Lemma~\ref{lem:probtreeball} below). They also ensure that the probability of the graph 
being simple has a strictly positive limit. Conditioned on being simple, the configuration model generates 
a random graph that is uniformly distributed among all the simple graphs with the given degree sequence 
(see van der Hofstad~\cite[Chapter 7]{RvdH1},~\cite[Chapters 3 and 6]{RvdH2}). Condition (R3) ensures 
that the random walk without backtracking is well-defined because it cannot get stuck on a dead-end.}
\end{remark}


\subsubsection{Dynamic configuration model}
\label{DCM}

We begin by describing the random graph process. It is convenient to take as the state space the set of 
configurations $\Conf_H$. For a fixed initial configuration $\eta$ and fixed $2 \leq k \leq m = \ell/2$, 
the graph evolves as follows (see Fig.~\ref{fig:dcm}):
\begin{enumerate}
\item
At each time $t \in \N$, pick $k$ edges (pairs of half-edges) from $C_{t-1}$ uniformly at random without 
replacement. Cut these edges to get $2k$ half-edges and denote this set of half-edges by $R_t$.
\item 
Generate a uniform pairing of these half-edges to obtain $k$ new edges. Replace the $k$ edges chosen 
in step 1 by the $k$ new edges to get the configuration $C_t$ at time $t$.
\end{enumerate}
This process rewires $k$ edges at each step by applying the configuration model sampling algorithm 
in Section~\ref{SCM} restriced to $k$ uniformly chosen edges. Since half-edges are not created or 
destroyed, the degree sequence of the graph given by $C_t$ is the same for all $t\in\N_0$. This gives 
us a Markov chain on the set of configurations $\Conf_H$. For $\eta,\zeta\in\Conf_H$, the \emph{transition 
probabilities} for this Markov chain are given by
\begin{align}
\label{DC}
Q(\eta,\zeta) = Q(\zeta,\eta) \coloneqq
\begin{cases}
\frac{1}{(2k-1)!!}\frac{\binom{m - d_\text{Ham}(\eta,\zeta)}{k - d_\text{Ham}(\eta,\zeta)}}{\binom{m}{k}} 
& \text{if }d_\text{Ham}(\eta,\zeta)\leq k, \\
0 
& \text{otherwise},
\end{cases}
\end{align}
where $d_\text{Ham}(\eta,\zeta) \coloneqq |\eta\setminus\zeta| = |\zeta\setminus\eta|$ is the Hamming 
distance between configurations $\eta$ and $\zeta$, which is the number of edges that appear in $\eta$ 
but not in $\zeta$. The factor $1/(2k-1)!!$ comes from the uniform pairing of the half-edges, while the 
factor $\binom{m - d_\text{Ham}(\eta,\zeta)}{k - d_\text{Ham}(\eta,\zeta)}/\binom{m}{k}$ comes 
from choosing uniformly at random a set of $k$ edges in $\eta$ that contains the edges in 
$\eta\setminus\zeta$. It is easy to see that this Markov chain is irreducible and aperiodic, with
stationary distribution the uniform distribution on $\Conf_H$, denoted by $\text{Conf}_H$, which 
is the distribution of the configuration model.

\begin{figure}[htbp]
\centering
\vspace{0.5cm}
\includegraphics[width=0.5\textwidth]{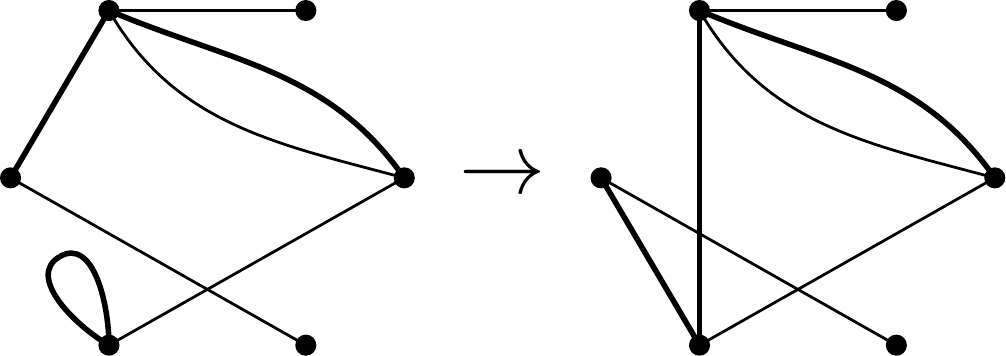}
\vspace{0.5cm}
\caption{\small One move of the dynamic configuration model. Bold edges on the left are the ones 
chosen to be rewired. Bold edges on the right are the newly formed edges.}
\label{fig:dcm}
\end{figure}


\subsubsection{Random walk without backtracking}
\label{NBTRW}

On top of the random graph process we define the random walk without backtracking, i.e., 
the walk cannot traverse the same edge twice in a row. As in Ben-Hamou and Salez~\cite{B-HS}, 
we define it as a random walk on the set of half-edges $H$, which is more convenient in the 
dynamic setting because the edges change over time while the half-edges do not.  For a fixed 
configuration $\eta$ and half-edges $x,y\in H$, the transition probabilities of the random walk 
are given by (recall \eqref{degdef})
\begin{align}
\label{NBT}
P_\eta(x,y) \coloneqq
\begin{cases}
\frac{1}{\deg(y)} & \text{if }\eta(x)\sim y\text{ and }\eta(x)\neq y, \\
0 & \text{otherwise}.
\end{cases}
\end{align}
When the random walk is at half-edge $x$ in configuration $\eta$, it jumps to one of the siblings 
of the half-edge it is paired to uniformly at random (see Fig.~\ref{fig:nbrw_move}). The transition 
probabilities are symmetric with respect to the pairing given by $\eta$, i.e., $P_\eta(x,y) 
= P_\eta(\eta(y),\eta(x))$, in particular, they are doubly stochastic, and so the uniform distribution 
on $H$, denoted by $U_H$, is stationary for $P_\eta$ for any $\eta\in\Conf_H$.

\begin{figure}[htbp]
\centering
\includegraphics[width=0.25\textwidth]{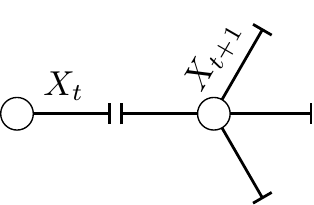}
\caption{\small The random walk moves from half-edge $X_t$ to half-edge $X_{t+1}$, one of 
the siblings of the half-edge that $X_t$ is paired to.}
\label{fig:nbrw_move}
\end{figure}

\subsubsection{Random walk on dynamic configuration model}

The random walk without backtracking on the dynamic configuration model is the joint Markov 
chain $(M_t)_{t\in\N_0} = (C_t,X_t)_{t\in\N_0}$ in which $(C_t)_{t\in\N_0}$ is the Markov chain 
on the set of configurations $\Conf_H$ as described in~\eqref{DC}, and $(X_t)_{t\in\N_0}$ 
is the random walk that at each time step $t$ jumps according to the transition probabilities 
$P_{C_t}(\cdot,\cdot)$ as in~\eqref{NBT}. 

Formally, for initial configuration $\eta$ and half-edge $x$, the one-step evolution of the joint 
Markov chain is given by the conditional probabilities
\begin{align}
\prob_{\eta,x}\big(C_t = \zeta, X_t = z \mid C_{t-1} = \xi, X_{t-1} = y\big) 
= Q(\xi,\zeta)\,P_\zeta(y,z),
\qquad t \in \N,
\end{align}
with
\begin{align}
\prob_{\eta,x}(C_0 = \eta,X_0=x) = 1.
\end{align}
It is easy to see that if $d(v)>1$ for all $v\in V$, then this Markov chain is irreducible and aperiodic, and 
has the unique stationary distribution $\text{Conf}_H \times U_H$. 

While the graph process $(C_t)_{t\in\N_0}$ and the joint process $(M_t)_{t\in\N_0}$ are Markovian, 
the random walk $(X_t)_{t\in\N_0}$ is not. However, $U_H$ is still the stationary distribution of 
$(X_t)_{t\in\N_0}$. Indeed, for any $\eta\in\Conf_H$ and $y \in H$ we have
\begin{equation}
\sum_{x\in H}U_H(x)\,\prob_{\eta,x}(X_t = y) 
= \sum_{x\in H} \frac{1}{\ell}\,\prob_{\eta,x}(X_t = y) = \frac{1}{\ell} = U_H(y).
\end{equation}
The next to last equality uses that $\sum_{x\in H}\prob_{\eta,x}(X_t = y)=1$ for every $y\in H$,
which can be seen by conditioning on the graph process and using that the space-time 
inhomogeneous random walk has a doubly stochastic transition matrix (recall the remarks 
made below \eqref{NBT}).


\subsection{Main theorem}
\label{statement}

We are interested in the behaviour of the total variation distance between the distribution of 
$X_t$ and the uniform distribution
\begin{equation}
\mathcal{D}_{\eta,x}(t) \coloneqq \|\prob_{\eta,x}(X_t\in\cdot\,)-U_H(\cdot)\|_{{\sss \mathrm{TV}}}.
\end{equation}
[We recall that the total variation distance of two probability measures $\mu_1,\mu_2$ on a finite 
state space $S$ is given by the following equivalent expressions:
\begin{equation}
\|\mu_1-\mu_2\|_{{\sss \mathrm{TV}}} \coloneqq \sum_{x\in S}|\mu_1(x)-\mu_2(x)| 
= \sum_{x\in S}[\mu_1(x)-\mu_2(x)]_+ = \sup_{A\subseteq S}[\mu_1(A)-\mu_2(A)],
\end{equation}
where $[a]_+ \coloneqq \max\{a,0\}$ for $a\in\R$.] Since $(X_t)_{t\in\N_0}$ is not Markovian, it is 
not clear whether $t\mapsto \mathcal{D}_{\eta,x}(t)$ is decreasing or not. On the other hand, 
\begin{equation}
\mathcal{D}_{\eta,x}(t) \leq \|\prob_{\eta,x}(M_t\in\cdot\,)-(U_H\times\text{Conf}_H)(\cdot)\|_{{\sss \mathrm{TV}}},
\end{equation}
and since the right-hand side converges to 0 as $t\to\infty$, so does $\mathcal{D}_{\eta,x}(t)$. 
Therefore the following definition is well-posed:

\begin{definition}[{\bf Mixing time of the random walk}]
\label{defMix}
For $\varepsilon\in(0,1)$, the {\em $\varepsilon$-mixing time of the random walk} is defined as
\begin{equation}
\label{mixing}
\tmix^n(\varepsilon; \eta, x) \coloneqq 
\inf\big\{t \in \N_0\colon\,\mathcal{D}_{\eta,x}(t) \leq \varepsilon\big\}.
\end{equation}
\end{definition}

Note that $\tmix^n(\varepsilon; \eta, x)$ depends on the initial configuration $\eta$ and half-edge $x$. 
We will prove statements that hold for \emph{typical} choices of $(\eta,x)$ under the uniform 
distribution $\mu_n$ (recall that $H$ depends on the number of vertices $n$) given by
\begin{equation}
\label{mu}
\mu_n := \text{Conf}_{H} \times U_H \quad \text{ on  }  \Conf_H \times H,
\end{equation} 
where {\em typical} is made precise through the following definition:

\begin{definition}[{\bf With high probability}]
\label{def:whp}
A statement that depends on the initial configuration $\eta$ and half-edge $x$ is said to hold 
{\em with high probability (whp)} in $\eta$ and $x$ if the $\mu_n$-measure of the set of pairs $(\eta,x)$ 
for which the statement holds tends to $1$ as $n\to\infty$. 
\end{definition}

\noindent
Below we sometimes write whp with respect to some probability measure other than $\mu_n$, 
but it will always be clear from the context which probability measure we are referring to.

Throughout the paper we assume the following condition on 
\begin{equation}
\label{ratio}
\alpha_n:=k/m, \qquad n \in \N,
\end{equation} 
denoting {\em the proportion of edges involved in the rewiring} at each time step of the graph 
dynamics defined in Section~\ref{DCM}:

\begin{cond}[{\bf Fast graph dynamics}]
\label{fast}
The ratio $\alpha_n$ in \eqref{ratio} is subject to the constraint 
\begin{equation}
\lim_{n\to\infty} \alpha_n (\log n)^2 = \infty.
\end{equation}
\end{cond}

We can now state our main result.

\begin{theorem}[{\bf Sharp mixing time asymptotics}]
\label{thm:mainthm}
Suppose that Conditions~{\rm \ref{cond-degree-reg}} and~{\rm \ref{fast}} hold. Then, for every
$\varepsilon > 0$, whp in $\eta$ and $x$, 
\begin{equation}
\label{mix}
\tmix^n(\varepsilon; \eta,x) = [1+o(1)]\,\sqrt{\frac{2\log(1/\varepsilon)}{\log(1/(1-\alpha_n))}}.
\end{equation}
\end{theorem}

\noindent
Note that Condition~\ref{fast} allows for $\lim_{n\to\infty} \alpha_n =0$. In that case \eqref{mix} 
simplifies to
\begin{equation}
\tmix^n(\varepsilon; \eta,x) = [1+o(1)]\,\sqrt{\frac{2\log(1/\varepsilon)}{\alpha_n}}.
\end{equation} 


\subsection{Discussion}

{\bf 1.}
Theorem~\ref{thm:mainthm} gives the sharp asymptotics of the mixing time in the regime 
where the dynamics is fast enough (as specified by Condition~\ref{fast}). Note that if 
$\lim_{n\to\infty} \alpha_n=\alpha \in (0,1]$, then $\tmix^n(\varepsilon; \eta, x)$ is of 
order one: at every step the random walk has a non-vanishing probability to traverse a 
rewired edge, and so it is qualitatively similar to a random walk on a complete graph. On 
the other hand, when $\lim_{n\to\infty} \alpha_n=0$ the mixing time is of order $1/\sqrt{\alpha_n} 
= o(\log n)$, which shows that the dynamics still \emph{speeds up} the mixing. The regime 
$\alpha_n = \Theta(1/(\log n)^2)$, which is not captured by Theorem~\ref{thm:mainthm}, 
corresponds to $1/\sqrt{\alpha_n} = \Theta(\log n)$, and we expect the mixing time to be 
\emph{comparable} to that of the static configuration model. In the regime $\alpha_n 
= o(1/(\log n)^2)$ we expect the mixing time to be the \emph{same} as that of the static 
configuration model. In a future paper we plan to provide a comparative analysis of the 
three regimes.
 
\medskip\noindent 
{\bf 2.}
In the static model the $\varepsilon$-mixing time is known to scale like $[1+o(1)]\,c\log n$
for some $c \in (0,\infty)$ that is independent of $\varepsilon \in (0,1)$ (Ben-Hamou and 
Salez~\cite{B-HS}). Consequently, there is \emph{cutoff}, i.e., the total variation distance 
drops from $1$ to $0$ in a time window of width $o(\log n)$. In contrast, in the regime 
of fast graph dynamics there is \emph{no cutoff}, i.e., the total variation distance drops 
from $1$ to $0$ gradually on scale $1/\sqrt{\alpha_n}$. 
 
\medskip\noindent
{\bf 3.}
Our proof is robust and can be easily extended to variants of our model where, for example, 
$(k_n)_{n\in\N}$ is random with $k_n$ having a first moment that tends to infinity as $n\to\infty$, 
or where time is continuous and pairs of edges are randomly rewired at rate $\alpha_n$.

\medskip\noindent
{\bf 4.} 
Theorem~\ref{thm:mainthm} can be compared to the analogous result for the static 
configuration model only when $\mathbb{P}_n(D_n \geq 3) = 1$ for all $n\in\N$. In fact, 
only under the latter condition does the probability of having a connected graph tend 
to one (see Luczak~\cite{Lucz92}, Federico and van der Hofstad~\cite{FedHof16}). 
If (R3) holds, then on the dynamic graph the walk mixes on the whole of $H$, while 
on the static graph it mixes on the subset of $H$ corresponding to the giant component.

\medskip\noindent
{\bf 5.}
We are not able to characterise the mixing time of the joint process of dynamic random 
graph and random walk. Clearly, the mixing time of the joint process is at least as large 
as the mixing time of each process separately. While the graph process helps the random 
walk to mix, the converse is not true because the graph process does not depend on the 
random walk. Observe that once the graph process has mixed it has an almost uniform 
configuration, and the random walk ought to have mixed already. This observation suggests 
that if the mixing times of the graph process and the random walk are not of the same 
order, then the mixing time of the joint process will have the same order as the mixing 
time of the graph process. Intuitively, we may expect that the mixing time of the graph 
corresponds to the time at which all edges are rewired at least once, which should be of 
order $(n/k)\log{n}=(1/\alpha_n)\log{n}$ by a coupon collector argument. In our setting 
the latter is much larger than $1/\sqrt{\alpha_n}$.

\medskip\noindent
{\bf 6.}
We emphasize that we look at the mixing times for `typical' initial conditions and we look 
at the distribution of the random walk averaged over the trajectories of the graph process:
the `annealed' model. It would be interesting to look at different setups, such as `worst-case' 
mixing, in which the maximum of the mixing times over all initial conditions is considered, 
or the `quenched' model, in which the entire trajectory of the graph process is fixed instead 
of just the initial configuration. In such setups the results can be drastically different. For 
example, if we consider the quenched model for $d$-regular graphs, then we see that for 
any time $t$ and any fixed realization of configurations up to time $t$, the walk without 
backtracking can reach at most $(d-1)^t$ half-edges. This gives us a lower bound of order 
$\log n$ for the mixing time in the quenched model, which contrasts with the $o(\log n)$ 
mixing time in our setup.

\medskip\noindent
{\bf 7.} 
It would be of interest to extend our results to random walk with backtracking, which is harder. 
Indeed, because the configuration model is locally tree-like and random walk without backtracking 
on a tree is the same as self-avoiding walk, in our proof we can exploit the fact that typical 
walk trajectories are self-avoiding. In contrast, for the random walk with backtracking, after it 
jumps over a rewired edge, which in our model serves as a randomized stopping time, it may 
jump back over the same edge, in which case it has not mixed. This problem remains to be 
resolved.


\subsection{Outline}

The remainder of this paper is organised as follows. Section~\ref{2} gives the main idea 
behind the proof, namely, we introduce a randomised stopping time $\tau=\tau_n$, the first 
time the walk moves along an edge that was rewired before, and we state a key proposition, 
Proposition~\ref{prop:mainprop} below, which says that this time is close to a strong stationary 
time and characterises its tail distribution. As shown at the end of Section~\ref{2}, 
Theorem~\ref{thm:mainthm} follows from Proposition~\ref{prop:mainprop}, whose proof 
consists of three main steps. The first step in Section~\ref{3} consists of a careful combinatorial 
analysis of the distribution of the walk given the history of the rewiring of the half-edges 
in the underlying evolving graph. The second step in Section~\ref{treelike} uses a classical 
exploration procedure of the static random graph from a uniform vertex to unveil the locally 
tree-like structure in large enough balls. The third step in Section~\ref{mainProof} settles 
the closeness to stationarity and provides control on the tail of the randomized stopping time
 $\tau$.


\section{Stopping time decomposition}
\label{2}

We employ a \emph{randomised stopping time} argument to get bounds on the total variation 
distance. We define the randomised stopping time $\tau=\tau_n$ to be the first time the walker 
makes a move through an edge that was rewired before. Recall from Section~\ref{DCM} that 
$R_t$ is the set of half-edges involved in the rewiring at time step $t$. Letting $R_{\leq t} = 
\cup_{s=1}^t R_s$, we set
\begin{equation}
\tau \coloneqq \min\{t\in\N\colon\, X_{t-1}\in R_{\leq t}\}.
\end{equation}
As we will see later, $\tau$ behaves like a strong stationary time. We obtain our main result by 
deriving bounds on $\mathcal{D}_{\eta,x}(t)$ in terms of conditional distributions of the random walk 
involving $\tau$ and in terms of tail probabilities of $\tau$. In particular, by the triangle inequality, for any 
$t\in\N_0$, $\eta\in\Conf_H$ and $x \in H$,
\begin{align}
\label{equ:tvuppertriangle}
\mathcal{D}_{\eta,x}(t) 
&\leq \prob_{\eta,x}(\tau > t)\,\|\prob_{\eta,x}(X_t\in\cdot\mid \tau >t) - U_H(\cdot)\|_{{\sss \mathrm{TV}}} \nn \\
&\qquad+ \prob_{\eta,x}(\tau \leq t)\,\|\prob_{\eta,x}(X_t\in\cdot\mid \tau \leq t) - U_H(\cdot)\|_{{\sss \mathrm{TV}}}
\end{align}
and
\begin{align}
\label{equ:tvlowertriangle}
\mathcal{D}_{\eta,x}(t) 
&\geq \prob_{\eta,x}(\tau > t)\,\|\prob_{\eta,x}(X_t\in\cdot\mid \tau >t) - U_H(\cdot)\|_{{\sss \mathrm{TV}}} \nn \\
&\qquad- \prob_{\eta,x}(\tau \leq t)\,\|\prob_{\eta,x}(X_t\in\cdot\mid \tau \leq t) - U_H(\cdot)\|_{{\sss \mathrm{TV}}}.
\end{align}
With these in hand, we only need to find bounds for $\prob_{\eta,x}(\tau > t)$, $\|\prob_{\eta,x}
(X_t\in\cdot\mid \tau >t) - U_H(\cdot)\|_{{\sss \mathrm{TV}}}$ and $\|\prob_{\eta,x}(X_t\in\cdot\mid \tau \leq t) 
- U_H(\cdot)\|_{{\sss \mathrm{TV}}}$.

The key result for the proof of our main theorem is the following proposition:

\begin{proposition}[\bf{Closeness to stationarity and tail behavior of stopping time}]
\label{prop:mainprop}
$\mbox{}$\\
Suppose that Conditions~{\rm \ref{cond-degree-reg}} and~{\rm \ref{fast}} hold. For $t = t(n)$ $ = o(\log n)$, 
whp in $x$ and $\eta$,
\begin{align}
&\|\prob_{\eta,x}(X_t\in\cdot\mid \tau\leq t) - U_H(\cdot)\|_{{\sss \mathrm{TV}}} = o(1), \label{equ:stoppedtv} \\
&\|\prob_{\eta,x}(X_t\in\cdot\mid \tau > t) - U_H(\cdot)\|_{{\sss \mathrm{TV}}} = 1-o(1), \label{equ:unstoppedtv} \\
&\prob_{\eta,x}(\tau > t) = (1-\alpha_n)^{t(t+1)/2}+o(1). \label{equ:tautailprob}
\end{align}
\end{proposition}

We close this section by showing how Theorem~\ref{thm:mainthm} follows from 
Proposition~\ref{prop:mainprop}:

\begin{proof}
By Condition~\ref{fast},
\begin{equation}
\sqrt{\frac{2\log(1/\varepsilon)}{\log(1/(1-\alpha_n))}} = O(\alpha_n^{-1/2}) = o(\log n).
\end{equation}
Using the bounds in~\eqref{equ:tvuppertriangle}--\eqref{equ:tvlowertriangle}, together 
with \eqref{equ:stoppedtv}--\eqref{equ:tautailprob} in Proposition~\ref{prop:mainprop}, 
we see that for $t=o(\log n)$,
\begin{equation}
(1-\alpha_n)^{t(t+1)/2} + o(1) \leq \mathcal{D}_{\eta,x}(t) \leq (1-\alpha_n)^{t(t+1)/2} + o(1).
\end{equation}
Choosing $t$ as in~\eqref{mix} we obtain $\mathcal{D}_{\eta,x}(t) = \varepsilon + o(1)$, 
which is the desired result.
\end{proof}

The remainder of the paper is devoted to the proof of Proposition~\ref{prop:mainprop}.


\section{Pathwise probabilities}
\label{3}

In order to prove \eqref{equ:stoppedtv} of Proposition~\ref{prop:mainprop}, we will show 
in~\eqref{BOUND} in Section~\ref{mainProof} that the following crucial bound holds for 
\emph{most} $y\in H$:
\begin{equation}
\label{equ:taustoppedtvlowerbound}
\prob_{\eta,x}(X_t = y\mid \tau\leq t) \geq \frac{1-o(1)}{\ell}.
\end{equation}
By \emph{most} we mean that the number of $y$ such that this inequality holds is $\ell - o(\ell)$ 
whp in $\eta$ and $x$. To prove~\eqref{equ:taustoppedtvlowerbound} we will look at $\prob_{\eta,x}
(X_t = y,\tau\leq t)$ by partitioning according to all possible paths taken by the walk and all possible 
rewiring patterns that occur on these paths. For a time interval $[s,t]\coloneqq\{s,s+1,\ldots,t\}$ 
with $s\leq t$, we define 
\begin{equation}
x_{[s,t]} \coloneqq x_s\cdots x_t.
\end{equation}
In particular, for any $y\in H$,
\begin{alignat}{2}
\label{equ:tstepwithstop}
&\prob_{\eta,x}(X_t = y, \tau\leq t) 
&&\\
&= \sum_{r = 1}^t\sum_{\substack{T\subseteq [1,t]\\|T|=r}}
\sum_{x_1,\dots,x_{t-1}\in H} \prob_{\eta,x}\Big(X_{[1,t]} = x_{[1,t]}, 
&& \,\, x_{i-1}\in R_{\leq i}\,\,\forall\, i\in T, 
\,\,x_{j-1}\not\in R_{\leq j}\,\,\forall\, j\in [1,t]\setminus T\Big)
\nn
\end{alignat}
with $x_0 = x$ and $x_t = y$. Here, $r$ is the number of steps at which the walk moves along a
previously rewired edge, and $T$ is the set of times at which this occurs.

For a fixed sequence of half-edges $x_{[0,t]}$ with $x_0 = x$ and a fixed set of times $T \subseteq [1,t]$ 
with $|T| = r$, we will use the short-hand notation 
\begin{equation}
A(x_{[0,t]}; T) \coloneqq \big\{x_{i-1} \in R_{\leq i}\,\,\forall\, i\in T, \,x_{j-1}\not\in R_{\leq j}\,\,\forall\, 
j\in [1,t] \setminus T\big\}.
\end{equation}
Writing $T = \{t_1,\dots,t_r\}$ with $1 \leq t_1 < t_2 < \dots < t_r \leq t$, we note that the conditional 
probability $\prob_{\eta,x}(X_{[1,t]} = x_{[1,t]}\mid A(x_{[0,t]}; T))$ can be non-zero only if each 
subsequence $x_{[t_{i-1},t_i-1]}$ induces a non-backtracking path in $\eta$ for $i \in [2,r+1]$ 
with $t_0 = 0$ and $t_{r+1} = t+1$. The last sum in~\eqref{equ:tstepwithstop} is taken over 
such sequences in $H$, which we call \emph{segmented paths} (see Fig.~\ref{segp}). For each 
$i \in [1,r+1]$ the subsequence $x_{[t_{i-1},t_i-1]}$ of length $t_i - t_{i-1}$ that forms a non-backtracking 
path in $\eta$ is called a \emph{segment}. 

\begin{figure}[htbp]
\centering
\includegraphics[width=0.4\textwidth]{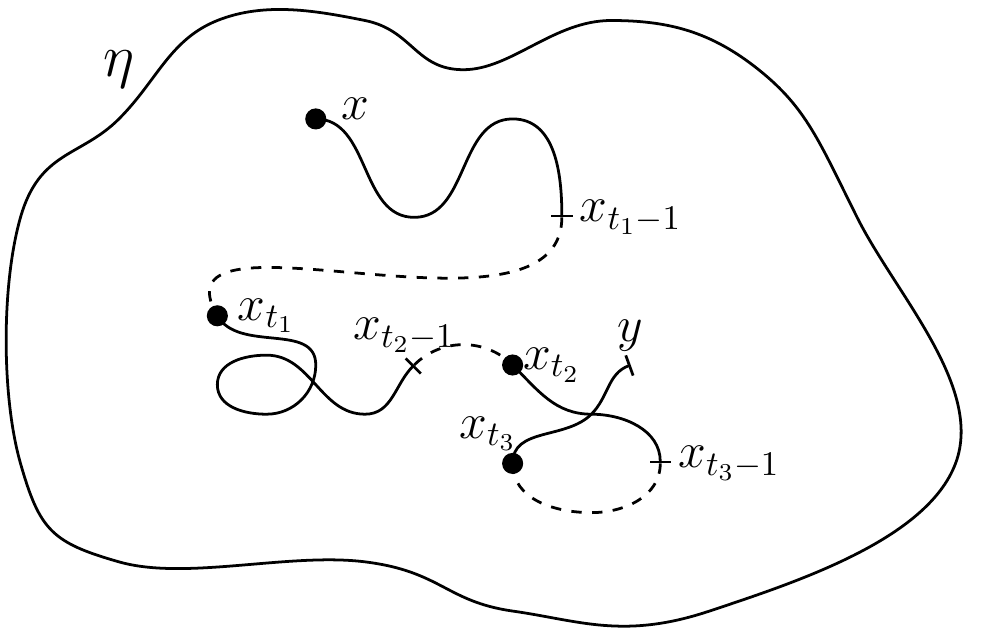}
\caption{\small An example of a segmented path with 4 segments. Solid lines represent the segments, 
consisting of a path of half-edges in $\eta$, dashed lines indicate the succession of the segments. 
The latter do not necessarily correspond to a pair in $\eta$, and will later correspond to rewired 
edges in the graph dynamics.}
\label{segp}
\end{figure}

We will restrict the last sum in~\eqref{equ:tstepwithstop} to the set of \emph{self-avoiding segmented paths}. 
These are the paths where no two half-edges are siblings, which means that the vertices $v(x_i)$ visited 
by the half-edges $x_i$ are distinct for all $i \in [0,t]$, so that if the random walk takes this path, then it does 
not see the same vertex twice. We will denote by $\SP_t^\eta(x,y;T)$ {\em the set of self-avoiding segmented 
paths} in $\eta$ of length $t+1$ that start at $x$ and end at $y$, where $T$ gives the positions of the 
ends of the segments (see Fig.~\ref{sasegp}). Segmented paths $x_{[0,t]}$ have the nice property that the 
probability $\prob_{\eta,x}(A(x_{[0,t]};T))$ is the same for all $x_{[0,t]}$ that are isomorphic, as stated
in the next lemma:

\begin{figure}[htbp]
\centering
\includegraphics[width=0.4\textwidth]{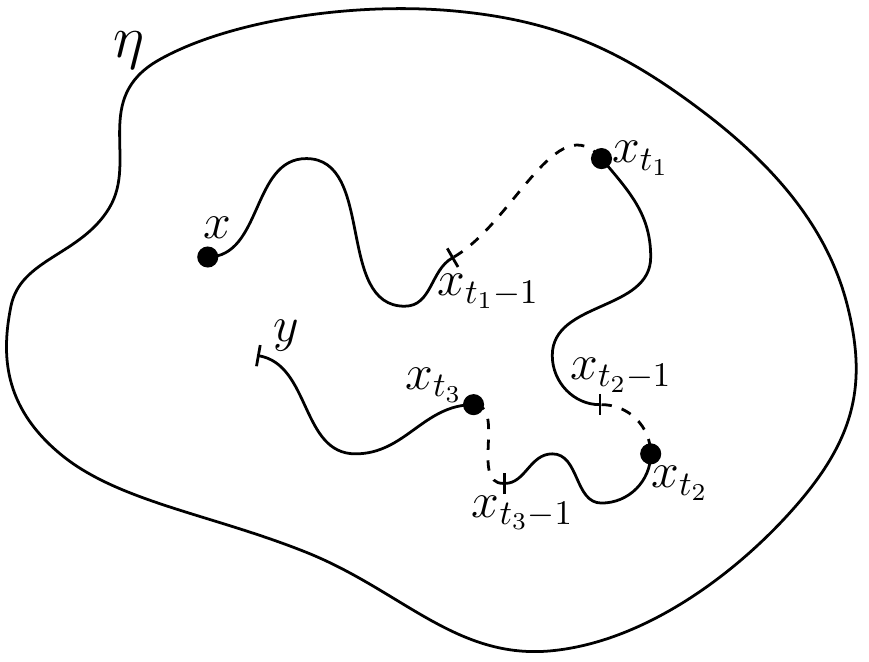}
\caption{\small An element of $\SP_t^\eta(x,y;T)$ with $T=\{t_1,t_2,t_3\}$.}
\label{sasegp}
\end{figure}

\begin{lemma}[\bf{Isomorphic segmented path are equally likely}]
\label{lem:equalpathprobs}
Fix $t\in\N$, $T\subseteq[1,t]$ and $\eta\in\Conf_H$. Suppose that $x_{[0,t]}$ and $y_{[0,t]}$ are two 
segmented paths in $\eta$ of length $t+1$ with $|x_{[s,s']}| = |y_{[s,s']}|$ for any  $0\leq s<s' \leq t$, 
where $|x_{[s,s']}|$ denotes the number of distinct half-edges in $x_{[s,s']}$. Then
\begin{equation}
\prob_{\eta,x}\big(A(x_{[0,t]};T)\big) = \prob_{\eta,x}\big(A(y_{[0,t]};T)\big).
\end{equation}
\end{lemma} 

\begin{proof}
Fix $x,y \in H$. Consider the coupling $((C_t^x)_{t\in\N_0},(C_t^y)_{t\in\N_0})$ of two dynamic 
configuration models with parameter $k$ starting from $\eta$, defined as follows. Let $f\colon\,
H\to H$ be such that
\begin{align}
f(x) =
\begin{cases}
y_i &\text{ if }x = x_i\text{ for some }i\in[0,t], \\
x_i &\text{ if }x = y_i\text{ for some }i\in[0,t],\\
\eta(y_i) &\text{ if }x = \eta(x_i)\text{ for some }i\in[0,t], \\
\eta(x_i) &\text{ if }x = \eta(y_i)\text{ for some }i\in[0,t], \\
x &\text{ otherwise.}
\end{cases}
\end{align}
This is a one-to-one function because $|x_{[s,s']}| = |y_{[s,s']}|$ for any  $0\leq s<s' \leq t$. What $f$ 
does is to map the half-edges of $x_{[0,t]}$ and their pairs in $\eta$ to the half-edges of $y_{[0,t]}$ 
and their pairs in $\eta$, and vice versa, while preserving the order in the path. For the coupling, at 
each time $t\in\N$ we rewire the edges of $C_{t-1}^x$ and $C_{t-1}^y$ as follows:
\begin{enumerate}
\item 
Choose $k$ edges from $C_{t-1}^x$ uniformly at random without replacement, say $\{z_1,z_2\},\dots,$ 
$\{z_{2k-1},z_{2k}\}$. Choose the edges $\{f(z_1),f(z_2)\},\dots,\{f(z_{2k-1}),f(z_{2k})\}$ from $C_{t-1}^y$.
\item 
Rewire the half-edges $z_1,\dots,z_{2k}$ uniformly at random to obtain $C_t^x$. Set $C_t^y(f(z_i)) 
= f(C_t^x(z_i))$.
\end{enumerate}
Step 2 and the definition of $f$ ensure that in Step 1 $\{f(z_1),f(z_2)\},\dots,\{f(z_{2k-1}),f(z_{2k})\}$ 
are in $C_{t-1}^y$. Since under the coupling the event $A(x_{[0,t]};T)$ is the same as the event 
$A(y_{[0,t]};T)$, we get the desired result.
\end{proof}

In order to prove the lower bound in~\eqref{equ:taustoppedtvlowerbound}, we will need two key facts. 
The first, stated in Lemma~\ref{lem:probsinglepath} below, gives a lower bound on the probability of 
a walk trajectory given the rewiring history. The second, stated in Lemma~\ref{lem:sizenicetuples} below,
is a lower bound on the number of relevant self-avoiding segmented paths, and exploits the locally 
tree-like structure of the configuration model. 

\begin{lemma}[{\bf Paths estimate given rewiring history}]
\label{lem:probsinglepath}
Suppose that $t = t(n) = o(\log n)$ and $T =\{t_1,\dots,t_r\}\subseteq [1,t]$. Let $x_0\cdots x_t\in
\SP_t^\eta(x,y;T)$ be a self-avoiding segmented path in $\eta$ that starts at $x$ and ends at $y$. 
Then
\begin{equation}
\prob_{\eta,x}\big(X_{[1,t]} = x_{[1,t]}\mid A(x_{[0,t]};T)\big) 
\geq \frac{1-o(1)}{\ell^r}\prod_{i\in[1,t]\setminus T}\frac{1}{\deg(x_i)}.
\end{equation} 
\end{lemma}

\begin{proof}
In order to deal with the dependencies introduced by conditioning on $A(x_{[0,t]};T))$, we will go through 
a series of conditionings. First we note that for the random walk to follow a specific path, the half-edges it 
traverses should be rewired \emph{correctly} at the \emph{right times}. Conditioning on $A(x_{[0,t]};T)$ 
accomplishes part of the job: since we have $x_{i-1}\not\in R_{\leq i}$ for $i\in [1,t]\setminus T$ and $x_{[0,t]} 
\in\SP_t^\eta(x,y;T)$, we know that, at time $i$, $x_{i-1}$ is paired to a sibling of $x_i$ in $C_i$, and so 
the random walk can jump from $x_{i-1}$ to $x_i$ with probability $1/\deg(x_i)$ at time $i$ for 
$i\in [1,t]\setminus T$. 

Let us call the path $x_{[0,t]}$ \emph{open} if $C_i(x_{i-1})\sim x_i$ for $i \in[1,t]$, i.e., if $x_{i-1}$ is 
paired to a sibling of $x_i$ in $C_i$ for $i \in [1,t]$. Then
\begin{equation}
\label{pathcondonopen}
\prob_{\eta,x}(X_{[1,t]} = x_{[1,t]}\mid x_{[0,t]}\text{ is open}) = \prod_{i=1}^t \frac{1}{\deg(x_i)},
\end{equation}
and
\begin{equation}
\prob_{\eta,x}\big(X_{[1,t]} = x_{[1,t]}\mid x_{[0,t]}\text{ is not open}\big) = 0.
\end{equation}
Using these observations, we can treat the random walk and the rewiring process separately, since the 
event $\{x_{[0,t]}\text{ is open}\}$ depends only on the rewirings. Our goal is to compute the probability
\begin{equation}
\prob_{\eta,x}\big(x_{[0,t]}\text{ is open}\mid A(x_{[0,t]};T)\big).
\end{equation}

Note that, by conditioning on $A(x_{[0,t]};T)$, the part of the path within segments is already open, so 
we only need to deal with the times the walk jumps from one segment to another. To have $x_{[0,t]}$ 
open, each $x_{t_j-1}$ should be paired to one of the siblings of $x_{t_j}$ for $j\in[1,r]$. Hence
\begin{align}
\label{openpathsum}
&\prob_{\eta,x}\big(x_{[0,t]}\text{ is open}\mid A(x_{[0,t]};T)\big) \nn \\
&\qquad= \sum_{\substack{z_1,\dots,z_r\in H\\z_j\sim x_{t_j}\,\forall\,j\in[1,r]}} 
\prob_{\eta,x}\big(C_{t_j}(x_{t_j-1})=z_j\,\,\forall\,j\in[1,r]\mid A(x_{[0,t]};T)\big).
\end{align}
Fix $z_1,\dots,z_r\in H$ with $z_j\sim x_{t_j}$, and let $y_j = x_{t_j-1}$ for $j\in[1,r]$. We will look at the 
probability
\begin{equation}
\prob_{\eta,x}\big(C_{t_j}(y_j)=z_j\,\,\forall\,j\in[1,r]\mid A(x_{[0,t]};T)\big).
\end{equation}

Conditioning on the event $A(x_{[0,t]};T)$ we impose that each $y_j$ is rewired at some time before $t_j$, 
but do not specify at which time this happens. Let us refine our conditioning one step further by 
specifying these times. Fix $s_1,\dots,s_r\in[1,t]$ such that $s_j\leq t_j$ for each $j\in[1,r]$ (the 
$s_j$ need not be distinct). Let $\widehat{A}$ be the event that $x_{i-1}\not\in R_{\leq i}$ for 
$i\in[1,t]\setminus T$ and $y_j$ is rewired at time $s_j$ for the last time before time $t_j$ for $j\in[1,r]$.
Then $\widehat{A}\subseteq A(x_{[0,t]};T)$. Since $s_j$ is the last time before $t_j$ at which $y_j$ is 
rewired, the event $C_{t_j}(y_j)=z_j$ is the same as the event $C_{s_j}(y_j)=z_j$ when we condition 
on $\widehat{A}$. We look at the probability
\begin{equation}
\prob_{\eta,x}\big(C_{s_j}(y_j)=z_j\,\,\forall\,j\in[1,r]\mid \widehat{A}\,\big).
\end{equation}
Let $s'_1<\dots<s'_{r'}\in[1,t]$ be the distinct times such that $s'_i = s_j$ for some $j\in[1,r]$, and $n_i^y$ 
the number of $j$'s for which $s_j = s'_i$ for $i\in[1,r']$, so that by conditioning on $\widehat{A}$ we 
rewire $n_i^y$ half-edges $y_j$ at time $s'_i$. Letting also $D_i = \{C_{s'_i}(y_j)=z_j,\text{ for } j
\text{ such that } s_j = s'_i\}$, we can write the above conditional probability as
\begin{equation}
\label{diprod}
\prod_{i=1}^{r'}\prob_{\eta,x}\big(D_i\mid \widehat{A},\,\cap_{j=1}^{i-1}D_j\big).
\end{equation}
We next compute these conditional probabilities.

Fix $i\in[1,r']$ and $\eta'\in\Conf_H$. We do one more conditioning and look at the probability
\begin{equation}
\prob_{\eta,x}\big(D_i\mid \widehat{A},\,\cap_{j=1}^{i-1}D_j,\,C_{s'_i-1}=\eta'\big).
\end{equation}
The rewiring process at time $s'_i$ consists of two steps: (1) pick $k$ edges uniformly at random;
(2) do a uniform rewiring. Concerning (1), by conditioning on $\widehat{A}$, we see that the $y_j$'s 
for which $s_j=s'_i$ are already chosen. In order to pair these to $z_j$'s with $s_j=s'_i$, the $z_j$'s 
should be chosen as well. If some of the $z_j$'s are already paired to some $y_j$'s already chosen, 
then they will be automatically included in the rewiring process. Let $m_i'$ be $m$ minus the number 
of half-edges in $\{x_0,\dots,x_t\}\cup\{z_1,\dots,z_r\}$, for which the conditioning on $\widehat{A}$ 
implies that they cannot be in $R_{s'_i}$. Then
\begin{align}
&\prob_{\eta,x}\Big(z_j\in R_{s'_i}\text{ for }j\text{ such that }s_j=s'_i \,~\Big|\,~\widehat{A},
\,\cap_{j=1}^{i-1}D_j,\,C_{s'_i-1}=\eta'\Big) \nn \\ 
&\geq \frac{\binom{m'_i-2n_i^y}{k-2n_i^y}}{\binom{m_i'-n_i^y}{k-n_i^y}}
= \frac{\prod_{j=0}^{n_i^y-1}(k-n_i^y-j)}{\prod_{j=0}^{n_i^y-1}(m'_i-n_i^y-j)} 
\geq \frac{\prod_{j=0}^{n_i^y-1}(k-n_i^y-j)}{m^{n_i^y}}.
\end{align}
Concerning (2), conditioned on the relevant $z_j$'s already chosen in (1), the probability that they will 
be paired to correct $y_j$'s is
\begin{equation}
\frac{1}{\prod_{j=1}^{n_i^y}(2k-2j+1)}.
\end{equation}
Since the last two statements hold for any $\eta'$ with $\prob_{\eta,x}(C_{s'_i-1}=\eta'\mid \widehat{A},\,
\cap_{j=1}^{i-1}D_j) > 0$, combining these we get
\begin{align}
\label{rol1}
\prob_{\eta,x}\big(D_i\mid \widehat{A},\,\cap_{j=1}^{i-1}D_j\big) 
&\geq \frac{\prod_{j=0}^{n_i^y-1}(k-n_i^y-j)}{m^{n_i^y}\prod_{j=1}^{n_i^y}(2k-2j+1)}
= \left(\frac{1-O(n_i^y/k)}{2m}\right)^{n_i^y}.
\end{align}
Since $\sum_{i=1}^{r'}n_i^y = r$, substituting \eqref{rol1} into~\eqref{diprod} and rolling back all the 
conditionings we did so far, we get
\begin{equation}
\prob_{\eta,x}\big(C_{t_j}(x_{t_j-1})=z_j\,\,\forall\,j\in[1,r]\mid A(x_{[0,t]};T)\big) 
\geq \frac{1-O(r^2/k)}{\ell^r} = \frac{1-o(1)}{\ell^r},
\end{equation}
where we use that $r^2/k\rightarrow 0$ since $r=o(\log{n})$ and $k=\alpha_n n$ with $(\log n)^2
\alpha_n\rightarrow \infty$. Now sum over $z_1,\dots,z_r$ in~\eqref{openpathsum}, to obtain
\begin{equation}
\prob_{\eta,x}\big(x_{[0,t]}\text{ is open}\mid A(x_{[0,t]};T)\big) 
\geq \frac{(1-o(1)) \prod_{j=1}^r \deg(x_{t_j})}{\ell^r},
\end{equation}
and multiply with~\eqref{pathcondonopen} to get the desired result.
\end{proof}


\section{Tree-like structure of the configuration model}
\label{treelike}

In this section we look at the structure of the neighborhood of a half-edge chosen uniformly at 
random in the configuration model. Since we will work with different probability spaces, we will 
denote by $\prob$ a generic probability measure whose meaning will be clear from the context.

For fixed $t\in\N$, $x\in H$ and $\eta\in\Conf_H$, we denote by $B_t^\eta(x):=\{y\in H\colon\,
\dist_{\eta}(x,y)\leq t\}$ the $t$-neighborhood of $x$ in $\eta$, where $\dist_{\eta}(x,y)$ is \emph{the 
length of the shortest non-backtracking path from $x$ to $y$}. We start by estimating the mean 
of $|B_t^\eta(x)|$, the number of half-edges in $B_t^\eta(x)$.

\begin{lemma}[{\bf Average size of balls of relevant radius}]
\label{lem:expecballsize}
Let  $\nu_n$ be as in Condition~{\rm \ref{cond-degree-reg}} and suppose that $t = t(n) = o(\log n)$.
Then, for any $\delta > 0$,
\begin{equation}
\expec(|B_t^\eta(x)|) = [1+o(1)]\,\nu_n^{t+1} = o(n^\delta),
\end{equation}
where the expectation is w.r.t.\ $\mu_n$ in~\eqref{mu}.
\end{lemma}

\begin{proof}
We have
\begin{equation}
|B_t^\eta(x)| = \sum_{y\in H}\mathds{1}_{\{\dist_{\eta}(x,y) \leq t\}}.
\end{equation}
Putting this into the expectation, we get
\begin{equation}
\label{equ:expecball}
\expec(|B_t^\eta(x)|) = \frac{1}{\ell}\sum_{x,y\in H} \prob(\dist_{\eta}(x,y) \leq t).
\end{equation}
For fixed $x,y\in H$,
\begin{align}
\prob(\dist_{\eta}(x,y)\leq t) 
&\leq \sum_{d=1}^t\sum_{x_1,\dots,x_{d-1}\in H}\prob(xx_1\cdots x_{d-1}y
\text{ forms a self-avoiding path in }\eta) \nn \\
&\leq \sum_{d=1}^t\sum_{x_1,\dots,x_{d-1}\in H}\left(\prod_{j=1}^{d-1}\frac{\deg(x_j)}{\ell-2j+1}\right)
\frac{\deg(y)}{\ell-2d+1} \nn \\
&= \frac{\deg(y)}{\ell}\sum_{d=1}^t\left(\prod_{i=1}^d\frac{\ell}{\ell-2i+1}\right)
\sum_{x_1,\dots,x_{d-1}\in H}\left(\prod_{i=1}^{d-1}\frac{\deg(x_i)}{\ell}\right) \nn \\
&= \frac{\deg(y)}{\ell}\sum_{d=1}^t\left(\prod_{i=1}^d\frac{\ell}{\ell-2i+1}\right)
\left(\sum_{z\in H}\frac{\deg(z)}{\ell}\right)^{d-1}.
\end{align}
Since $t = o(\log n)$ and $\ell = \Theta(n)$, we have
\begin{equation}
\prob(\dist_{\eta}(x,y)\leq t) \leq [1+o(1)]\,\frac{\deg(y)}{\ell}\,(\nu_n)^t.
\end{equation}
Substituting this into~\eqref{equ:expecball}, we get
\begin{equation}
\expec(|B_t^\eta(x)|) \leq \frac{1+o(1)}{\ell}\sum_{x,y\in H} \frac{\deg(y)}{\ell}\,(\nu_n)^t 
= [1+o(1)]\,(\nu_n)^{t+1} = o(n^\delta),
\end{equation}
where the last equality follows from (R2) in Condition~\ref{cond-degree-reg} and the fact that 
$t = o(\log n)$.
\end{proof}

For the next result we will use an \emph{exploration process} to build the neighborhood of a uniformly 
chosen half-edge. (Similar exploration processes have been used in \cite{B-HS},\cite{BLPS} and 
\cite{LS}.) We explore the graph by starting from a uniformly chosen half-edge $x$ and building up 
the graph by successive uniform pairings, as explained in the procedure below. Let $\G(s)$ denote 
the \emph{thorny graph} obtained after $s$ pairings as follows (in our context, a thorny graph is a 
graph in which half-edges are not necessarily paired to form edges, as shown in Fig.~\ref{fig:Gpic}). 
We set $\G(0)$ to consist of $x$, its siblings, and the incident vertex $v(x)$. Along the way we keep 
track of \emph{the set of unpaired half-edges at each time $s$}, denoted by $U(s)\subset H$, and 
the so-called \emph{active} half-edges, $A(s) \subset U(s)$. We initialize $U(0) = H$ and $A(0) 
= \{x\}$. We build up the sequence of graphs $(\G(s))_{s\in\N_0}$ as follows:
\begin{enumerate}
\item 
At each time $s\in\N$, take the \emph{next} unpaired half-edge in $A(s-1)$, say $y$. Sample a half-edge 
uniformly at random from $H$, say $z$. If $z$ is already paired or $z = y$, then reject and sample again. 
Pair $y$ and $z$.
\item 
Add the newly formed edge $\{y,z\}$, the incident vertex $v(z)$ of $z$, and its siblings to $\G(s-1)$, to 
obtain $\G(s)$.
\item 
Set $U(s) = U(s-1)\setminus\{y,z\}$, i.e., remove $y,z$ from the set of unpaired half-edges, and set 
$A(s) = A(s-1)\cup\{H(v(z))\}\setminus\{y,z\}$, i.e., add siblings of $z$ to the set of active half-edges 
and remove the active half-edges just paired.
\end{enumerate}
This procedure stops when $A(s)$ is empty. We think of $A(s)$ as a first-in first-out queue. So, when 
we say that we pick the \emph{next} half-edge in Step 1, we refer to the half-edge on top of the 
queue, which ensures that we maintain the breadth-first order. The rejection sampling used in Step 1 
ensures that the resulting graph is distributed according to the configuration model. This procedure 
eventually gives us the connected component of $x$ in $\eta$, the part of the graph that can be 
reached from $x$ by a non-backtracking walk, where $\eta$ is distributed uniformly on $\Conf_H$. 

\begin{figure}[htbp]
\centering
\includegraphics[width=0.5\textwidth]{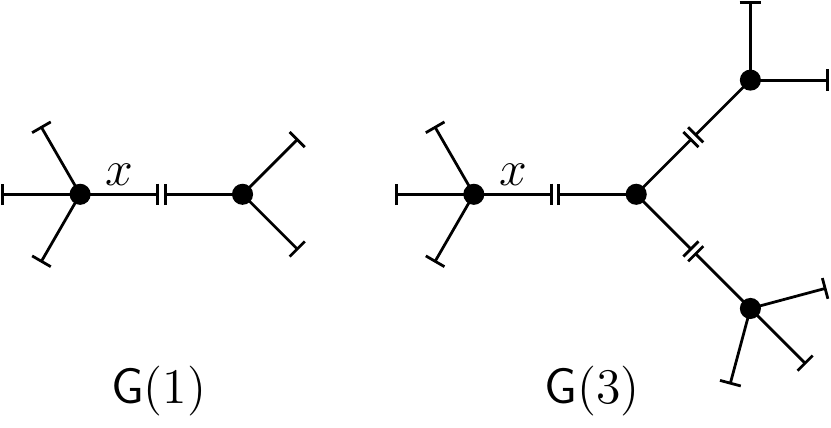}
\caption{\small Example snapshots of $\G(s)$ at times $s=1$ and $s=3$.}
\label{fig:Gpic}
\end{figure}

\begin{lemma}[{\bf Tree-like neighborhoods}]
\label{lem:probtreeball}
Suppose that $s = s(n) = o(n^{(1-2\delta)/2})$ for some $\delta \in (0, \tfrac12)$. Then $\G(s)$ is a tree with 
probability $1-o(n^{-\delta})$.
\end{lemma}

\begin{proof}
Let $F$ be the first time the uniform sampling of $z$ in Step 1 fails at the first attempt, or $z$ is a 
sibling of $x$, or $z$ is in $A(s-1)$. Thus, at time $F$ we either choose an already paired half-edge 
or we form a cycle by pairing to some half-edge already present in the graph. We have
\begin{equation}
\prob(\G(s)\text{ is not a tree}) \leq \prob(F \leq s).
\end{equation}
Let $Y_i$, $i \in \N$, be i.i.d.\ random variables whose distribution is the same as the distribution of 
the degree of a uniformly chosen half-edge. When we form an edge before time $F$, we use one 
of the unpaired half-edges of the graph, and add new unpaired half-edges whose number is distributed 
as $Y_1$. Hence the number of unpaired half-edges in $\G(u)$ is stochastically dominated by 
$\sum_{i=1}^{u+1}Y_i-u$, with one of the $Y_i$'s coming from $x$ and the other ones coming from 
the formation of each edge. Therefore the probability that one of the conditions of $F$ will be met 
at step $u$ is stochastically dominated by $(\sum_{i=1}^{u}Y_i + u -2)/\ell$. We either choose an 
unpaired half-edge in $\G(u)$ or we choose a half-edge belonging to an edge in $\G(u)$, and by 
the union bound we have
	\eqan{
	&\prob(\G(s)\text{ is not a tree}\mid (Y_i)_{i\in[1,s]}) \leq \prob(F\leq s\mid (Y_i)_{i\in[1,s]}) \nn \\
	&\leq \frac{\sum_{u=1}^{s} \sum_{i=1}^{u} (Y_i + u -2)}{\ell}
	=\frac{\sum_{i=1}^{s} (s-i+1)Y_i + s(s-1)/2}{\ell}.
	}
Since $\expec(Y_1) = \nu_n = O(1)$ and $s = o(n^{(1-2\delta)/2})$, via the Markov inequality we get
that, with probability at least $1-o(n^{-\delta})$,
	\begin{equation}
	s\sum_{i=1}^{s} Y_i < n^{1-\delta}.
	\end{equation}
Combining this with the bound given above and the fact that $\ell=\Theta(n)$, we arrive at
\begin{equation}
\prob(\G(s)\text{ is not a tree}) = o(n^{-\delta}). 
\end{equation}
\end{proof}

To further prepare for the proof of the lower bound in \eqref{equ:taustoppedtvlowerbound} and 
Proposition~\ref{prop:mainprop} in Section~\ref{mainProof}, we introduce one last ingredient. 
For $x\in H$ and $\eta\in\Conf_H$, we denote by $\bar{B}_t^\eta(x)$ the set of half-edges from 
which there is a non-backtracking path to $x$ of length at most $t$. For fixed $t\in\N$, 
$T = \{t_1,\dots,t_r\} \subseteq [1,t]$ and $\eta\in\Conf_H$, we say that an $(r+1)$-tuple 
$(x_0,x_1,\dots,x_r)$ is \emph{good} for $T$ in $\eta$ if it satisfies the following two properties:
\begin{enumerate}
\item 
$B_{t_{j}-t_{j-1}}^\eta(x_j)$ is a tree for $j \in [1,r]$ with $t_0 = 0$, and $\bar{B}_{t-t_r}^\eta(x_r)$ 
is a tree.
\item 
The trees $B_{t_{j}-t_{j-1}}^\eta(x_j)$ for $j \in [1,r]$ and $\bar{B}_{t-t_r}^\eta(x_r)$ are all 
disjoint.
\end{enumerate}
For a good $(r+1)$-tuple all the segmented paths, such that the $i$th segment starts from $x_{i-1}$ 
and is of length $t_i-t_{i-1}$ for $i\in[1,r]$ and the $(r+1)$st segment ends at $x_r$ and is of 
length $t-t_r$, are self-avoiding by the tree property. The next lemma states that whp in $\eta$ 
almost all $(r+1)$-tuples are good. We denote by $N_t^\eta(T)$ the set of $(r+1)$-tuples that are 
good for $T$ in $\eta$, and let $N_t^\eta(T)^c$ be the complement of $N_t^\eta(T)$. We have the 
following estimate on $|N_t^\eta(T)|$:

\begin{lemma}[{\bf Estimate on good paths}]
\label{lem:sizenicetuples}
Suppose that $t = t(n) = o(\log n)$. Then there exist $\bar\delta>0$ such that whp in $\eta$ for all 
$T\subseteq[1,t]$,
\begin{equation}
\label{nt}
|N_t^\eta(T)| = (1-n^{-\bar\delta})\ell^{|T|+1}.
\end{equation}
\end{lemma}

\begin{proof}
Fix $\varepsilon > 0$ and $T\subseteq[1,t]$ with $|T|=r$. We want to show that whp $|N_t^\eta(T)^c| 
\leq \varepsilon\ell^{r+1}$. By the Markov inequality, we have
\begin{equation}
\prob(|N_t^\eta(T)^c| > \varepsilon\ell^{r+1}) \leq \frac{\expec(|N_t^\eta(T)^c|)}{\varepsilon\ell^{r+1}} 
= \frac{\prob(Z_{[0,r]}\in N_t^\eta(T)^c)}{\varepsilon},
\end{equation}
where $Z_0,\dots,Z_r$ are i.i.d.\ uniform half-edges and we use that $1/\ell^{r+1}$ is the uniform 
probability over a collection of $r+1$ half-edges. Let $B_{i-1} = B_{t_{i}-t_{i-1}}^\eta(Z_{i-1})$ 
for $i \in [1,r]$ and $B_r = B_{t-t_r}^\eta(Z_r)$. By the union bound,
\begin{align}
\prob\big(Z_{[0,r]} \in N_t^\eta(T)^c\big) 
&\leq \sum_{i = 0}^r\prob(B_i\text{ is not a tree}) + \sum_{i,j=0}^r\prob(B_i\cap B_j\neq\varnothing).
\end{align}

By Lemma~\ref{lem:expecballsize} and since $t = o(\log n)$, for any $0<\delta<\tfrac12$ we have 
$\expec|B_i| = o(n^\delta)$, and so by the Markov inequality $|B_i| = o(n^{(1-2\delta)/2})$ with 
probability $1-o(n^{-\delta})$. Hence, by Lemma~\ref{lem:probtreeball} and since $\ell = \Theta(n)$, 
for $i \in [1,r]$, we have
\begin{equation}
\prob(B_{i-1}\text{ is not a tree}) = o(n^{-\delta}).
\end{equation}
Again using Lemma~\ref{lem:expecballsize}, we see that for any $i,j\in[1,r]$,
\begin{align}
\prob(B_i\cap B_j\neq\varnothing) \leq \prob(Z_j\in B_t^\eta(Z_i)) 
= \frac{\expec(|B_t^\eta(Z_i)|)}{\ell} \leq o(n^{\delta-1}).
\end{align}
Since $r \leq t = o(\log n)$, setting $\bar\delta = 2\delta/3$ and taking $\varepsilon = n^{-\delta}$, 
we get
\begin{equation}
\prob(|N_t^\eta(T)^c| > \varepsilon\ell^{r+1}) \leq \frac{rn^{-\bar\delta} + r^2n^{\bar\delta-1}}{\varepsilon} 
= o(n^{-\delta/4})
\end{equation}
uniformly in $T\subseteq[1,t]$. Since there are $2^t$ different $T\subseteq[1,t]$ and $2^t = 2^{o(\log n)} 
= o(n^{\delta/8})$, taking the union bound we see that~\eqref{nt} holds for all $T\subseteq[1,t]$ with 
probability $1-o(n^{-\delta/8})$.
\end{proof}


\section{Closeness to stationarity and tail behavior of stopping time}
\label{mainProof}

We are now ready to prove the lower bound in~\eqref{equ:taustoppedtvlowerbound} and 
Proposition~\ref{prop:mainprop}. Before giving these proofs, we need one more lemma, for 
which we introduce some new notation. For fixed $t\in\N$, $T\subseteq[1,t]$ with $|T| = r >0$, 
$\eta\in\Conf_H$ and $x,y\in H$, let $N_t^\eta(x,y;T)$ denote the set of $(r-1)$-tuples such 
that $(x,x_1,\dots,x_{r-1},y)$ is good for $T$ in $\eta$. Furthermore, for a given $(r+1)$-tuple 
$\mathbf{x} = (x,x_1,\dots,x_{r-1},y)$ that is good for $T$ in $\eta$, let $\SP_t^\eta(\mathbf{x};T)$ 
denote the set of all segmented paths in which the $i$th segment starts at $x_{i-1}$ and is of 
length $t_i-t_{i-1}$ for $i\in[1,r]$ with $x_0 = x$ and $t_0 = 0$, and the $(r+1)$st segment ends 
at $y$ and is of length $t-t_r$. By the definition of a good tuple, these paths are self-avoiding, 
and hence $\SP_t^\eta(\mathbf{x};T) \subset \SP_t^\eta(x,y;T)$.

\begin{lemma}[{\bf Total mass of relevant paths}]
\label{lem:lowerboundsumsspaths}
Suppose that $t = t(n) = o(\log n)$. Then whp in $\eta$ and $x,y$ for all $T\subseteq [1,t]$, 
\begin{equation}
\sum_{x_{[0,t]}\in\SP_t^\eta(x,y;T)}\prob_{\eta,x}\big(X_{[1,t]} = x_{[1,t]}\mid A(x_{[0,t]};T)\big) 
\geq \frac{1-o(1)}{\ell}.
\end{equation}
\end{lemma}

\begin{proof}
By Lemma~\ref{lem:sizenicetuples}, the number of pairs of half-edges $x,y$ for which $|N_t^\eta(x,y;T)| 
\geq (1-n^{-\bar\delta})\ell^{|T|-1} = [1-o(1)]\,\ell^{|T|-1}$ for all $T\in[1,t]$ is at least $(1-2^tn^{-\bar\delta})
\ell^2 = [1-o(1)]\,\ell^2$ whp in $\eta$. Take such a pair $x,y\in H$, and let $r = |T|$. By 
Lemma~\ref{lem:probsinglepath} and the last observation before the statement of 
Lemma~\ref{lem:lowerboundsumsspaths}, we have
\begin{align}
&\sum_{x_{[0,t]}\in\SP_t^\eta(x,y;T)}\prob_{\eta,x}\big(X_{[1,t]} = x_{[1,t]}\mid A(x_{[0,t]};T)\big) \nn \\
&\qquad \geq \sum_{\mathbf{x}\in N_t^\eta(x,y;T)}\sum_{y_0\dots y_t\in\SP_t^\eta(\mathbf{x},T)}
\frac{1-o(1)}{\ell^r}\prod_{i\in[1,t]\setminus T}\frac{1}{\deg(y_i)}.
\end{align}

We analyze at the second sum by inspecting the contributions coming from each segment separately. 
For fixed $\mathbf{x}\in  N_t^\eta(x,y;T)$, when we sum over the segmented paths in $\SP_t^\eta
(\mathbf{x},T)$, we sum over all paths that go out of $x_{i-1}$ of length $t_{i}-t_{i-1}$ for $i\in[1,r]$. 
Since $\prod_{j=t_{i-1}+1}^{t_{i}-1}\tfrac{1}{\deg(y_j)}$ is the probability that the random walk without
backtracking follows this path on the static graph given by $\eta$ starting from $x_{i-1}$, when we 
sum over all such paths the contribution from these terms sums up to 1 for each $i \in[1,r]$, i.e., the 
contributions of the first $r$ segments coming from the products of inverse degrees sum up to 1.  
For the last segment we sum, over all paths going into $y$, the probability that the random walk 
without backtracking on the static graph given by $\eta$ follows the path. Since the uniform distribution 
is stationary for this random walk, the sum over the last segment of the probabilities $\tfrac{1}{\ell}
\prod_{j= t_r+1}^t\tfrac{1}{\deg(y_j)}$ gives us $1/\ell$. With this observation, using that $|N_t^\eta(x,y;T)|
\geq (1-o(1))\ell^{r-1}$, we get
\begin{align}
&\sum_{x_{[0,t]}\in\SP_t^\eta(x,y;T)}\prob_{\eta,x}\big(X_{[1,t]} = x_{[1,t]}\mid A(x_{[0,t]};T)\big) \nn \\
&\qquad \geq \frac{1-o(1)}{\ell}\sum_{\mathbf{x}\in N_t^\eta(x,y;T)}\frac{1-o(1)}{\ell^{r-1}} = \frac{1-o(1)}{\ell},
\end{align}
which is the desired result.
\end{proof}

\begin{proof}[$\bullet$ Proof of~\eqref{equ:stoppedtv}]
For any self-avoiding segmented path $x_0\cdots x_t$, we have $|x_{[s,s']}| = s'-s+1$ for all $1\leq s<s'\leq t$. 
By Lemma~\ref{lem:equalpathprobs}, the probability $\prob_{\eta,x}(A(x_{[0,t]};T))$ depends on $\eta$ 
and $T$ only, and we can write $\prob_{\eta,x}(A(x_{[0,t]};T)) = p_t^\eta(T)$ for any $xx_1\cdots x_{t-1}y\in
\SP_t^\eta(x,y;T)$. Applying Lemma~\ref{lem:lowerboundsumsspaths}, we get
\begin{align}
&\prob_{\eta,x}(X_t = y,\tau\leq t)\\ \nn
&\geq \sum_{r=1}^t\sum_{\substack{T\subseteq [1,t]\\|T|=r}}
\sum_{x_{[0,t]}\in\SP_t^\eta(x,y;T)} \prob_{\eta,x}\big(X_{[1,t]} = x_{[1,t]}\mid A(x_{[0,t]};T)\big)
\,\prob_{\eta,x}\big(A(x_{[0,t]};T)\big) \\ \nn
&\geq \frac{1-o(1)}{\ell}\sum_{r=1}^t\sum_{\substack{T\subseteq [1,t]\\|T|=r}}p_t^\eta(T).
\end{align}
If the $t$-neighborhood of $x$ in $\eta$ is a tree, then all $t$-step non-backtracking paths starting at $x$ 
are self-avoiding. (Here is a place where the non-backtracking nature of our walk is crucially used!)
In particular, for any such path $xx_1\cdots x_t$ we have $\prob_{\eta,x}(A(x_{[0,t]};\varnothing)) 
= p_t^\eta(\varnothing)$. Denoting by $\Gamma_t^\eta(x)$ the set of paths in $\eta$ of length $t$ 
that start from $x$, we also have
\begin{align}
\prob_{\eta,x}(\tau > t) 
&= \sum_{x_0\cdots x_t\in\Gamma_t^\eta(x)}\prob_{\eta,x}\big(X_{[1,t]} = x_{[1,t]}, A(x_{[0,t]};\varnothing)\big) \nn \\
&= \sum_{x_0\cdots x_t\in\Gamma_t^\eta(x)}\prod_{i=1}^t \frac{1}{\deg(x_i)}\,p_t^\eta(\varnothing) = p_t^\eta(\varnothing),
\end{align}
since the product $\prod_{i=1}^t \frac{1}{\deg(x_i)}$ is the probability that a random walk without 
backtracking in the static $\eta$ follows the path $x_0x_1\cdots x_t$, and we take the sum over 
all paths going out of $x$.

For a fixed path $x_0x_1\cdots x_t$, we have
\begin{align}
\sum_{r=1}^t\sum_{\substack{T\subseteq [1,t]\\|T|=r}}\prob_{\eta,x}\big(A(x_{[0,t]};T)\big) 
= 1 - \prob_{\eta,x}\big(A(x_{[0,t]};\varnothing)\big).
\end{align}
So, when the $t$-neighborhood of $x$ in $\eta$ is a tree, we have
\begin{align}
\sum_{r=1}^t\sum_{\substack{T\subseteq [1,t]\\|T|=r}}p_t^\eta(T) 
= 1 - p_t^\eta(\varnothing) = 1 - \prob_{\eta,x}(\tau > t) = \prob_{\eta,x}(\tau \leq t),
\end{align}
which gives
\begin{equation}
\label{BOUND}
\prob_{\eta,x}(X_t = y,\tau\leq t) \geq \frac{1-o(1)}{\ell}\,\prob_{\eta,x}(\tau \leq t)
\end{equation}
and settles the lower bound~\eqref{equ:taustoppedtvlowerbound}. Since the latter holds whp in $\eta$ 
and $x,y$, we have that the number of $y$ for which this holds is $[1-o(1)]\,\ell$ whp in $\eta$ and $x$. 
Denoting the set of $y \in H$ for which the lower bound in~\eqref{equ:taustoppedtvlowerbound} holds 
by $N_t^\eta(x)$, we get that whp in $\eta$ and $x$,
\begin{align}
\|\prob_{\eta,x}(X_t\in\cdot\mid\tau\leq t) - U_H(\cdot)\|_{\TV} 
&= \sum_{y\in H}\left[\frac{1}{\ell} - \prob_{\eta,x}(X_t=y\mid\tau\leq t)\right]^+ \nn \\
&\leq \sum_{y\in N_t^\eta(x)}\left[\frac{1}{\ell} - \frac{1-o(1)}{\ell}\right]^+ 
+ \sum_{y\not\in N_t^\eta(x)}\frac{1}{\ell} = o(1),
\end{align}
which is~\eqref{equ:stoppedtv}.
\end{proof}

\begin{proof}[$\bullet$ Proof of~\eqref{equ:unstoppedtv}]
First note that $\prob_{\eta,x}(X_t\in B_t^\eta(x)\mid \tau >t) = 1$. On the other hand, using 
Lemma~\ref{lem:expecballsize} and the Markov inequality, we see that $U_H(B_t^\eta(x)) 
= |B_t^\eta(x)|/\ell = o(1)$ whp in $\eta$ and $x$, and so we get
\begin{equation}
\|\prob_{\eta,x}(X_t\in\cdot \mid \tau > t) - U_H(\cdot)\|_{\TV} 
\geq \prob_{\eta,x}(X_t\in B_t^\eta(x)\mid \tau >t) - U_H(B_t^\eta(x)) = 1-o(1).
\end{equation} 
\end{proof}

\begin{proof}[$\bullet$ Proof of \eqref{equ:tautailprob}]
Taking $T = \varnothing$ in Lemma~\ref{lem:sizenicetuples}, we see that $B_t^\eta(x)$ is a tree whp 
in $\eta$ and $x$, so each path in $\eta$ of length $t$ that goes out of $x$ is self-avoiding. By looking 
at pathwise probabilities, we see that
\begin{align}
\label{equ:tautailprobsum}
\prob_{\eta,x}(\tau >t) = \sum_{x_0\cdots x_t\in\Gamma_t^\eta(x)}\
\prob_{\eta,x}\big(X_{[1,t]} = x_{[1,t]}, x_{i-1}\not\in R_{\leq i}\,\forall\,i\in[1,t]\big).
\end{align}
Since the event $\{x_{i-1}\not\in R_{\leq i}\,\forall\,i\in[1,t]\}$ implies that the edge involving $x_{i-1}$ 
is open a time $i$, 
\begin{align}
\label{equ:singlepathprob}
\prob_{\eta,x}\big(X_{[1,t]} = x_{[1,t]}\mid x_{i-1}\not\in R_{\leq i}\,\forall\,i\in[1,t]\big) 
= \prod_{i=1}^t \frac{1}{\deg(x_i)}.
\end{align}
Next, let us look at the probability $\prob_{\eta,x}(x_i\not\in R_{\leq i}\,\forall\,i\in[1,t])$. 
By rearranging and conditioning, we get
\begin{align}
&\prob_{\eta,x}\big(x_{i-1}\not\in R_{\leq i}\,\forall\,i\in[1,t]\big) 
= \prob_{\eta,x}\big(x_j\not\in R_i\,\forall\,j\in[i-1,t-1]\,\forall\,i\in[1,t]\big) \nn \\
&\qquad= \prod_{i=1}^t \prob_{\eta,x}\Big(x_j\not\in R_i\,\forall j\in[i-1,t-1] 
~\Big|~ x_k\not\in R_j\,\forall\,k\in[j-1,t-1]\,\forall\,j\in[1,i-1]\Big).
\end{align}
Observe that, on the event $\{x_k\not\in R_j\,\forall\,k\in[j-1,t-1\,\forall\,j\in[1,i-1]\}$, the path 
$x_{i-1}\cdots x_{t-1}$ has not rewired until time $i-1$, and so the number of edges given by 
these half-edges is $t-i+1$, since it was originally a self-avoiding path. With this we see that 
for any $i\in[1,t]$,
\begin{equation}
\prob_{\eta,x}\Big(x_j\not\in R_i\,\forall j\in[i-1,t-1] ~\Big|~ 
x_k\not\in R_j\,\forall\,k\in[j-1,t-1]\,\forall\,j\in[1,i-1]\Big) 
= \frac{\binom{m - t + i -1}{k}}{\binom{m}{k}}, 
\end{equation}
and hence
\begin{align}
\prob_{\eta,x}\big(x_{i-1}\not\in R_{\leq i}\,\forall\,i\in[1,t]\big) 
&= \prod_{i=1}^t \frac{\binom{m - t + i -1}{k}}{\binom{m}{k}} = \prod_{i=1}^t \frac{\binom{m - i}{k}}{\binom{m}{k}} \nn \\
&= \prod_{i=1}^t\prod_{j=0}^{i-1}\left(1-\frac{k}{m-j}\right) = \prod_{j=1}^t\left(1-\frac{k}{m-j+1}\right)^{t-j+1}.
\end{align}
Since $j \leq t = o(\log n)$, $m = \Theta(n)$ and $n/\log^2n= o(k)$, we have
\begin{equation}
\prob_{\eta,x}\big(x_{i-1}\not\in R_{\leq i}\text{ for all }i\in[1,t]\big) 
= [1+o(1)]\,(1-k/m)^{t(t+1)/2} = (1-\alpha_n)^{t(t+1)/2} + o(1).
\end{equation}
Putting this together with~\eqref{equ:singlepathprob} and inserting it into~\eqref{equ:tautailprobsum}, we get
\begin{align}
\prob_{\eta,x}(\tau > t) = [(1-\alpha_n)^{t(t+1)/2} + o(1)]\sum_{x_0\cdots x_t\in\Gamma_t^\eta(x)}
\prod_{i=1}^t \frac{1}{\deg(x_i)} = (1-\alpha_n)^{t(t+1)/2} + o(1),
\end{align}
since, for each path $x_0\cdots x_t$, the product $\prod_{i=1}^t \frac{1}{\deg(x_i)}$ is the probability that 
the random walk without backtracking on the static graph given by $\eta$ follows the path, and we sum 
over all paths starting from $x$.
\end{proof}


\end{document}